\documentclass[12pt]{amsart}
\usepackage{amssymb}
\usepackage{amsfonts}
\usepackage{amscd}
\usepackage[T1]{fontenc}

\swapnumbers

\newcommand{\Q}{\mathbb{Q}}
\newcommand{\F}{\mathbb{F}}
\newcommand{\Z}{\mathbb{Z}}
\newcommand{\N}{\mathbb{N}}

\def\11{{\mathbf 1}}

\def\cK{{\mathcal K}}
\def\cL{{\mathcal L}}
\def\cM{{\mathcal M}}

\def\cO{{\mathcal O}}

\mathchardef\alphag="7C0B
\mathchardef\betag="7C0C
\mathchardef\gammag="7C0D
\mathchardef\deltag="7C0E
\mathchardef\varepsilong="7C22
\mathchardef\varphig="7C27
\mathchardef\psig="7C20
\mathchardef\zetag="7C10
\mathchardef\epsilong="7C0F
\mathchardef\rhog="7C1A
\mathchardef\taug="7C1C
\mathchardef\upsilong="7C1D
\mathchardef\iotag="7C13
\mathchardef\thetag="7C12
\mathchardef\pig="7C19
\mathchardef\sigmag="7C1B
\mathchardef\etag="7C11
\mathchardef\omegag="7C21
\mathchardef\kappag="7C14
\mathchardef\lambdag="7C15
\mathchardef\mug="7C16
\mathchardef\xig="7C18
\mathchardef\chig="7C1F
\mathchardef\nug="7C17
\mathchardef\varthetag="7C23
\mathchardef\varpig="7C24
\mathchardef\varrhog="7C25
\mathchardef\varsigmag="7C26
\mathchardef\Omegag="7C0A
\mathchardef\Thetag="7C02
\mathchardef\Sigmag="7C06
\mathchardef\Deltag="7C01
\mathchardef\Phig="7C08
\mathchardef\Gammag="7C00
\mathchardef\Psig="7C09
\mathchardef\Lambdag="7C03
\mathchardef\Xig="7C04
\mathchardef\Pig="7C05
\mathchardef\Upsilong="7C07

\newtheorem{thm}{Theorem}
\newtheorem{lem}{Lemma}
\newtheorem*{cor*}{Corollary}

\theoremstyle{remark}
\newtheorem{remark}{Remark}

\theoremstyle{plain}

\numberwithin{equation}{subsection}

\def\boxit#1#2{\setbox1=\hbox{\kern#1{#2}\kern#1}%
\dimen1=\ht1 \advance\dimen1 by #1
\dimen2=\dp1 \advance\dimen2 by #1
\setbox1=\hbox{\vrule height\dimen1 depth\dimen2\box1\vrule}%
\setbox1=\vbox{\hrule\box1\hrule}%
\advance\dimen1 by .4pt \ht1=\dimen1
\advance\dimen2 by .4pt \dp1=\dimen2 \box1\relax}

  {\par\medskip\noindent #1\par\begingroup%
    \advance\leftskip by 1em\advance\rightskip by 1em}%
  {\par\endgroup}

\begin{document}

\setcounter{tocdepth}{2} 

\title[Finite by Presburger Groups and $p$-adic Fields]{Model theory of finite-by-Presburger Abelian groups and finite 
extensions of $p$-adic fields}


\author[Jamshid Derakhshan]{Jamshid Derakhshan}
\address{University of Oxford, Mathematical Institute, Andrew Wiles Building, Radcliffe Observatory Quarter, Woodstock Road, Oxford, OX2 6GG, UK}
\email{derakhsh@maths.ox.ac.uk}

\author[Angus Macintyre]{Angus Macintyre}
\address{Queen Mary, University of London,
School of Mathematical Sciences, Queen Mary, University of London, Mile End Road, London E1 4NS, UK}
\email{angus@eecs.qmul.ac.uk}

\subjclass[2000]{Primary 03C10, 03C60, 11D88, 11U09; Secondary 11U05}
\keywords{model theory, $p$-adic numbers, local fields, model completeness, quantifier elimination, pre-ordered groups, Presburger arithmetic}

\begin{abstract} 
We define a class of pre-ordered abelian groups that we call 
finite-by-Presburger groups, and prove that 
their theory is model-complete. We show that certain quotients of the multiplicative group 
of a local field of characteristic zero are finite-by-Presburger and interpret the higher residue rings of the local field. 
We apply these results to give a new proof of the model completeness in the ring language of a local field of characteristic zero (a result that follows also from work of Prestel-Roquette).
\end{abstract}

\maketitle

\section{Introduction}

A theory $T$ is called {\it model-complete} if for any model $M$ of $T$ and any $n\geq 1$, 
any definable subset of $M^n$ is 
defined by an existential formula. This concept was defined by 
Abraham Robinson (cf. \cite{Robinson-book}).

In this paper we define a class of pre-ordered abelian groups and prove that their theory is model-complete. Given a local field of 
characteristic zero $K$, we show that certain quotients of the multiplicative group $K^*$ are finite-by-Presburger.  We also show that they interpret
the higher residue rings of the local field and other structure from the Basarab-Kuhlman language for valued fields. As an application of these results, 
we give a new proof of model completeness for a finite extension of a $p$-adic field $\Q_p$ (a result that also follows from work of Prestel-Roquette) via 
result on first-order definitions of the valuation rings.

\section{Finite-by-Presburger Abelian groups}\label{sec-pres}

We consider the language of group theory with primitives $\{.,1,^{-1}\}$, together with a symbol $\leq$ 
standing for pre-order. The 
intended structures are abelian groups $G$, equipped with a binary relation $\leq$ satisfying
$$\forall g~(g\leq g),$$
$$\forall g\forall h\forall j~(g\leq h \wedge h\leq j \Rightarrow g\leq j),$$
$$\forall g\forall h~(g\leq h \vee h\leq g),$$
$$\forall g\forall h\forall j~(g\leq h \Rightarrow gj\leq hj).$$
It would be natural to call such structures {\it pre-ordered abelian groups}. 

Define $g \sim h$ to mean $g\leq h$ and $h\leq g$. This is obviously a congruence on $G$, and the quotient 
$G/\sim$ is naturally an ordered abelian group. {\it We restrict to the case when 
$\{g: g\sim 1\}$ is a finite group $H$}. We call such $G$ {\it finite-by-ordered}. Note that the projection map
$$G\rightarrow G/\sim$$
is pre-order preserving.

\begin{lem} $H$ is the torsion subgroup of $G$ if $G$ is finite-by-ordered.\end{lem}
\begin{proof} $G/\sim$ is torsion free.\end{proof}
Note that $H$ is pure in $G$, indeed, if $g\in G$ satisfies $g^m\in H$ for some $m$, then $g\in H$. By 
\cite[Theorem 7,pp.18]{Kaplansky}, a pure subgroup of bounded exponent in an abelian group is a direct summand. 
Clearly $H$ is of bounded exponent (being finite!), so $H$ is a direct factor of $G$, so $G=H.\Gamma$, 
an internal direct product of subgroups, for some $\Gamma$. 

Now $\Gamma$ contains at most one element from each $\sim$-class, and the relation $\leq$ on $\Gamma$ gives $\Gamma$ the structure of an 
ordered abelian group. So in fact since
$G$ is the product of two pre-ordered groups, one of which $H$ has only one $\sim$-class. So $\Gamma \cong G/H$ as ordered 
abelian groups. 

Since $G$ is a direct product of two pre-ordered groups, we have the following.

\begin{thm} The theory of $(G,\leq)$ is determined by the theory of $H$ and the theory of the ordered group $(G/H,\leq)$. 
Moreover, $G$ is decidable if and only if $(G/H,\leq)$ is decidable.
\end{thm}
\begin{proof} Follows from the Feferman-Vaught Theorem \cite{FV}.\end{proof}

We would like model-completeness of $(G,\leq)$ but settle here for a special case when $G/H$ is a model of 
Presburger arithmetic. Now Presburger arithmetic has 
quantifier elimination in the language with primitives $\{.,1,^{-1},\tau,P_n, \leq\}$, where $.$ 
denotes multiplication, $\tau$ is a constant interpreted 
as the minimal positive element, $\leq$ is an ordering, and $P_n$ is the subgroup of $n$th powers. Note that this is the 
multiplicative version of the usual formalism 
of Presburger arithmetic (cf.~\cite[Section 3.2, pp.197]{Enderton}).

So we augment the basic formalism of pre-ordered abelian groups 
with symbols $\tau$ and $P_n$, for all $n\geq 2$ as above, and to the axioms of 
pre-ordered groups we add the following set of axioms for any given finite group $H$. 
(In these axioms $m$ denotes the exponent of $H$, and $Tor(G)$ the torsion subgroup of $G$. 

i) If the relation $\leq$ is an order, then $\tau$ is the minimal positive element, 
and if not, then $\tau=1$. 

ii) If $g\in G$ and $g$ has order $k$ for some $k\in \N$, then $k$ divides $m$ (we have a sentence for each 
$k\geq 1$).

iii) $Tor(G)\models \sigma$, where $\sigma$ denote a sentence that 
characterizes the group $H$ up to isomorphism (note that this sentence 
exists since $H$ is finite).

iv) If $g\in G$ satisfies $g\sim 1$, then $g\in H$.

v) $G/T$ is totally ordered and is a model of Presburger arithmetic with $\tau H$ the minimal positive element. 

vi) The order $\leq$ on $H$ is trivial (i.e.~for any two $g,h\in H$ we have $g\leq h$ and $h\leq g$).

Note that given a model $\cM$ of these axioms, 
$H$ is the isomorphic to the torsion subgroup of $\cM$ (by (iii)). 
Thus, given any finite group $H$, we obtain a theory which we denote by 
$\mathcal{T}_H$. Note that if $H=1$ (the identity group!), then $\mathcal{T}_H$ is the theory 
of Presburger arithmetic. We call these the axioms of pre-ordered groups with torsion $H$ and ordered Presburger 
quotient modulo $H$.

Clearly $G$ from above enriches to a model of these axioms. 

\begin{thm}\label{model-completeness-group} The theory determined by the above axioms is model-complete. 
It follows that $(G,\leq)$ is 
model-complete.
\end{thm}
\begin{proof} Let $M_1\rightarrow M_2$ be an embedding of models of the above axioms. 
We know as above that
$$M_2=H.\Gamma_2$$
for some $\Gamma_2$. Let $\Gamma_1:=\Gamma_2\cap M_1$. Then we have
$$M_1=H.\Gamma_1.$$
Thus the embedding $M_1\rightarrow M_2$ is the product embedding
$$H.\Gamma_1 \rightarrow H.\Gamma_2.$$
Now $H\rightarrow H$ is elementary (indeed, take $\gamma=1$ in both copies of $H$), and
$$\Gamma_1\rightarrow \Gamma_2$$
is elementary since the map
$$M_1/H \rightarrow M_2/H$$
is elementary because both ordered groups have the same minimal positive element.
Therefore by the Feferman-Vaught Theorem \cite{FV} the map
$$H.\Gamma_1\rightarrow H.\Gamma_2$$
is elementary.
\end{proof}

\section{Groups of additive and multiplicative congruence classes}

Let $K$ be a valued field. We shall denote by $\cO_K$ and $\cM_K$ the valuation ring and the valuation ideal respectively. We assume that 
$K$ has residue characteristic $p>0$. We denote the value group of $K$ by $\Gamma$. 
For an integer $k\geq 0$, set 
$$\cM_{K,k}=\{a\in \cM_K: v(a)>kv(p)\},$$
$$\cO_{K,k}=\cO_K/\cM_{K,k},$$
a local ring, and 
$$G_{K,k}=K^*/1+\cM_{K,k},$$
a multiplicative group. $\pi_{k}$ denotes the canonical projection
$$\cO_K \rightarrow \cO_{K,k},$$
and $\pi_{k}^*$ the canonical projection 
$$K^* \rightarrow G_{K,k}.$$
We denote by
$$\Theta_k\subseteq G_{K,k} \times \cO_{K,k}$$
the binary relation defined by 
$$\Theta_{k}(x,y) \Leftrightarrow \exists z \in \cO_K (\pi^*_{k}(z)=x \wedge \pi_{k}(z)=y).$$
We denote by $\cK_{k}$ the many-sorted structure 
$$(K,G_{K,k},\cO_{K,k},\Theta_{k}).$$
Note that $v$ is well-defined on $G_{K,k}$ and surjective to the value group $\Gamma$. 

The groups $G_{K,k}$ are called the groups of multiplicative congruences and the rings $\cO_{K,k}$ are called the higher residue rings of $K$. They occured in the 
work of Hasse on local fields. In model theory they first appeared in the language of Basarab \cite{basarab} and then simplified by Kuhlmann \cite{Kuhlmann}.  His works with 
the many-sorted language
$$(\cL_{rings},\cL_{groups},\cL_{rings},\pi_k,\pi^*_k,\Theta_{k}),$$
for local fields. This has a sort for the field $K$ equipped with the language of rings, a sort for the groups 
$G_{K,k}$ equipped with the language of groups $\cL_{groups}$, and a sort for the 
residue rings $\cO_{K,k}$ equipped with the language of rings, for all $k\geq 0$. The 
language has symbols for the projection maps 
$\pi_k$ and $\pi^*_k$ and a predicate for the relation $\Theta_{k}$. We call this the 
language of Basarab-Kuhlmann and denote it by $\cL_{BK}$. 

Note that $\cL_{BK}$ does not have a symbol for the valuation on $K$ and on $G_{K,k}$. However the valuation is 
quantifier-free definable from $\Theta_k$.

\begin{lem} Let $K$ be a finite extension of $\Q_p$ where $p$ is a prime. For any $k$, the groups $G_{K,k}$ are pre-ordered $H$-Presburger, where $H$ is the torsion group of $G_{K,k}$.
\end{lem}
\begin{proof}
We first identify the torsion elements 
of $G_{K,k}$. Clearly these must be of the form $g(1+\cM_{K,k})$ where $v(g)=0$. Note that 
$$g^{p^f-1}\in 1+\cM_{K}$$
and
$$(g^{p^f-1})^{p^k}\in 1+\cM_{K,k}.$$ 
Thus $g$ has 
(in $G_{K,k}$) order dividing $(p^{f}-1)p^k$, and if
$$g\in 1+\cM_{K},$$
then $g$ has order dividing $p^k$ in $G_{K,k}$. 
Thus the torsion subgroup of $G_{K,k}$ has order $(p^f-1)(p^f)^{ke}$. If $U$ denotes the group 
of units of $\cO_K$. Then $H:=U/1+\cM_{K,k}$ is the torsion subgroup of $G_{K,k}$. Thus
$G_{K,k}/H$ is isomorphic to $K^*/U$ which is the value group of $K$, and hence is a $\Z$-group, and so a model of Presburger arithmetic. \end{proof}

\begin{thm} For any $k$, the rings $\cO_{K,k}$ and the relation $G_{K,k}$ are interpretable in $G_{K,k}$.\end{thm}
\begin{proof}
Let $\pi$ denote an element of least positive value in $K_1$ (it follows that $\pi$ is also an element of least 
positive value in $K_2$). We let $\mu$ denote a generator of the cyclic group consisting of the Teichmuller representatives in 
$K_1$ (and hence the same holds for $\mu$ in $K_2$). $\mu$ has order $p^f-1$. As before we have 
$k=ef$ where $f$ and $e$ are respectively the residue 
field degree and ramification index of $L$ over $\Q_p$.

An element of $\cO_{K_1,k}$ can be written uniquely in the form
$$a+\cM_{K_1,k},$$
where $a\in K$ can be uniquely 
represented as
$$\sum_{0\leq j\leq k} c_j\pi^j$$
where $c_j$ are either $0$ or a power of $\mu$. Similarly, an element of 
$\cO_{K_2,k}$ is uniquely of the form $a+\cM_{K_2,k}$. Now except when all $c_j=0$, these elements map to elements 
of $G_{K_i,k}$ (where $i=1,2$) under the map
$$(\sum_{0\leq j\leq k}c_j\pi^j+\cM_{K_i,k}) \rightarrow (\sum_{0\leq j\leq k} c_j\pi^j)(1+\cM_{K_i,k}).$$
This map is injective. Indeed, if two elements 
$\sum_{0\leq j\leq k}c_j\pi^j$ and $\sum_{0\leq j\leq k}c'_j\pi^j$ map to the same element, then their difference 
lies in $\cM_{K_i,k}$, but if $\gamma_1$ and $\gamma_2$ are different powers of $\mu$, then 
$v(\gamma_1-\gamma_2)=0$ by the usual Hensel Lemma argument that gives us the Teichmuller set, this gives a 
contradiction. 

So we may construe the {\it nonzero} elements $\sum_{0\leq j\leq k}c_j\pi^j+\cM_{K_1,k}$ 
as constant elements of $G_{K_1,k}$ 
(and the same for $G_{K_2,k}$). We shall use the 
notation
$$[\sum_{0\leq j\leq k}c_j\pi^j+\cM_{K_1,k}]$$
for them (similarly for $G_{K_2,k}$).
We have a multiplication on these elements coming from the group $G_{K_i,k}$, for $i=1,2$, 
which we denote by $\odot$. It is defined by 
$$[r_1]\odot [r_2]=[r_1].[r_2],$$
where $.$ is group multiplication in $G_{K_i,k}$. We also have an addition on these elements together with the zero 
element $0$ coming from the ring 
$\cO_{K_i,k}$, for $i=1,2$, which we denote by $\oplus$. It is defined by 
$$[r_1]\oplus [r_2]=[r_1+r_2].$$

We thus have a finite subset, denoted by $R_1$ (resp.~ $R_2$), of $G_{K_1,k}$ (resp.~ $G_{K_2,k}$) consisting of the nonzero 
elements
$$[\sum_{0\leq j\leq k}c_j\pi^j+\cM_{K_1,k}]$$
(resp.~ $[\sum_{0\leq j\leq k}c_j\pi^j+\cM_{K_1,k}]$) 
above together with the operations $\oplus,\odot$ satisfying 
$$([r_1]\oplus [r_2])\odot [r_3]=[r_1]\odot [r_1] \oplus [r_1]\odot [r_3],$$
and the properties that $[1]$ is the unit element of $\odot$ and $[\pi^{k+1}]$ is the zero element. 

Now, for $i=1,2$, using Lemma \ref{theta}, 
we can interpret in $G_{K_i,k}$ the relation $\Theta_k$ as the set $\Theta_k^{+}$ of all pairs 
$(g,r)\in G_{K_i,k}\times R_i$ satisfying the formula
$$(r=[\pi^{k+1}] \wedge v(g)\geq k+1) \vee$$
$$\bigvee_{s} (0\leq v(g) \leq k \wedge v([s])=v(g) \wedge r=[s])),$$
where $s$ runs through the nonzero elements $\sum_{0\leq j\leq k}c_j\pi^j+\cM_{K_i,k}$ 
from before. (In fact, the $s$ satisfying the above is unique). Thus 
$$G_{K_i,k}\times R$$ with the relation $\Theta_k^{+}$ as above and with factors the two sorts is 
isomorphic to the structure $$G_{K_i,k}\times \cO_{K_i,k}$$ with the relation $\Theta_k$ and with factors the two sorts.\end{proof}


One has the following result of Basarab-Kuhlman on quantifier elimination.
\begin{thm}\cite{Kuhlmann}\label{bk-thm} Let $K$ be a Henselian valued field with characteristic zero and 
residue characteristic $p>0$. Then given an $\cL_{BK}$-formula $\varphi(\bar x)$, 
there is an $\cL_{BK}$-formula $\psi(\bar x)$ which is quantifier free in the field sort such that for all $\bar x$
$$K\models \varphi(\bar x) \Leftrightarrow \cK_k \models \psi(\bar x).$$
\end{thm}

Note that for $k=0$, $\cO_{K,k}$ is the residue field, and $G_{K,k}$ comes with an exact sequence
$$1 \rightarrow k^* \rightarrow G_{K,0} \rightarrow \Gamma \rightarrow 1.$$
We shall need a suitable description of the relation $\Theta_k$ as follows.
\begin{lem}\label{theta} For any valued field $K$ and $k\geq 0 $,
$$\Theta_k=\{(g,\alpha)\in G_{K,k} \times \cO_{K,k}: 
(\alpha=0 \wedge v(g)\geq k+1)\vee (\alpha \neq 0 \wedge v(g)\leq k)\}.$$
\end{lem}
\begin{proof}
Obvious.
\end{proof}

\section{First-order definitions of valuation rings of local fields} 

We shall denote by $\cL_{rings}$ the (first-order) 
language of rings with primitives $\{+,.,0,1\}$. Given a structure $K$, we let 
$Th(K)$ denote the $\cL_{rings}$-theory of $K$, i.e., the set of all $\cL_{rings}$-sentences that are true in $K$. 

Let $L$ be a finite extension of $\Q_p$, where $p$ is a prime. By a theorem of F.K.~Schmidt (cf.~\cite[Theorem 4.4.1]{EngPres}), any two Henselian valuation rings of 
$L$ are comparable, so since $L$ has a rank 1 valuation, it has a unique valuation ring $\cO_L$ giving a 
Henselian valuation. By \cite[Theorem 6]{CDLM}, 
this valuation ring is defined by an existential $\cL_{rings}$-formula $\psi(x)$. We remark that $\psi(x)$ depends on 
the field $L$. For any field $K$ which is elementarily equivalent to $L$, $\psi(x)$ defines a valuation ring in $K$ 
and hence a valuation.


By Krasner's Lemma (see \cite[Section 1]{CF}), $L=\Q_p(\delta)$ for some $\delta$ algebraic over $\Q$, and 
$L$ has only finitely many extensions of each finite dimension. This property (with the same numbers) is true for any 
$K$ which satisfies $K\equiv L$. 

From the $\Sigma_1$-definability of $\cO_L$ we easily get a $\Sigma_1$-definition of the set
$$\{x: v(x)\leq 0\},$$
and of the set of units $\{x: v(x)=0\}$. 
But it seems that no general nonsense argument gives a $\Sigma_1$-definition of the maximal ideal $\{x: v(x)>0\}$. 

We shall be working throughout in the language of rings, and our structures and morphisms and formulas 
are from this language unless otherwise stated.  

Note that it is a necessary condition for model-completeness that 
$$\cO_{K_2}\cap K_1=\cO_{K_1},$$
whenever 
$K_1\rightarrow K_2$ is an embedding of models of $Th(L)$. We shall establish this condition for all embeddings of models 
of $Th(L)$. For this, we shall first prove the following lemma.

\begin{lem}\label{lem2} Let $K_1 \rightarrow K_2$ be an embedding of models of $Th(L)$. Then
\begin{enumerate}
 \item $K_1$ is relatively algebraically closed in $K_2$,
 \item The valuation induced from $\cO_{K_2}$ on $K_1$ is Henselian.
\end{enumerate}
\end{lem}
\begin{proof} We first give a proof of (1). 
Suppose $n=[L:\Q_p]$. Then $n=ef$, where $e$ is the ramification index and $f$ the residue field dimension 
(see \cite{EngPres},\cite{CF}). Clearly it is a first-order (but not yet visibly existential) 
property of $\cO_L$ (defined by $\psi(x)$) expressed in the language of rings  
that the residue field has $p^f$ elements. Thus 
both $K_1$ and $K_2$ have residue fields (with respect to $\cO_{K_1}$ and $\cO_{K_2}$) of cardinality $p^f$. 
(Recall, of course, 
that we do not yet know \ref{(*)}, so we have no natural map of residue fields). Similarly, in both 
$K_1$ and $K_2$ we have that $v(p)$ is the $e$th positive element of the value group (a condition that can be expressed by a first-order 
sentence using the formula $\psi(x)$ defining the valuation).

We now argue by contradiction. 
Suppose $K_1$ is not relatively algebraically closed in $K_2$, then $K_1(\beta) \subset K_2$, for some $\beta$ which is 
algebraic over $K_1$ of degree $m>1$. The valuation $v$ of $K_1$ defined by $\psi(x)$ has a unique extension $w$ to $K_1(\beta)$ by Henselianity and 
\cite[Theorem 4.4.1]{EngPres}. We have that 
$m=e'f'$, where $e'$ is the ramification index and $f'$ is the residue field dimension of $K_1(\beta)$ over $K_1$ with respect to $w$. 
($L$ satisfies all such equalities 
and so $K_1$ does too. All this is of course with respect to the topology defined by $\psi(x)$). Now if $f'>1$ we 
may replace $K_1(\beta)$ by its maximal subfield unramified over $K_1$. So we can in that case assume  
$K_1(\beta)$ is unramified over $K_1$. Now $K_1$ has residue field $\F_{p^f}$, and then by Hensel's Lemma 
$K_1(\beta)$ contains a primitive $(p^{ff'}-1)$th root of unity (similar arguments are used in 
\cite{CDLM}). So $K_2$ contains a primitive $(p^{ff'}-1)$th root of unity. But $K_2$ certainly does not, since it's residue 
field (with respect to $\psi(x)$) is $\F_{p^f}$ also.

So we must have $f'=1$, i.e.~$K_1(\beta)$ is totally ramified over $K_1$. Now we can assume that $\beta$ is a 
root of a monic 
Eisenstein (relative to $\cO_{K_1}$) polynomial $F(x)$ over $K_1$. Let 
$$F(x)=x^{e'}+c_1x^{e'-1}+\dots+c_{e'}.$$
Note that $F(x)$ can not be Eisenstein over $K_2$, for then it would be irreducible, and it has a root $\beta$ in $K_2$.

Within $K_1$ the condition that $c_j$ is in the maximal ideal (for $\cO_{K_1}$!) is simply that 
$$c_j^{e}p^{-1}\in \cO_{K_1},$$
and the condition that $c_{e'}$ is a uniformizing element is simply that both 
$$c_{e'}^{e}p^{-1}\in \cO_{K_1},$$
and 
$$c_{e'}^{-e}p\in \cO_{K_1},$$
hold. Now these conditions go up into $K_2$ since $\psi(x)$ is a $\Sigma_1$-formula.
So
$$c_j^ep^{-1}\in \cO_{K_2}$$
for all $1\leq j\leq e'$, and 
$$c_{e'}^{-e}p\in \cO_{K_2}.$$
Now $v(p)$ (in the sense of $\cO_{K_2}$) is the $e$th positive element of the value group (true in $L$). So in fact each 
$v(c_j)>0$ (in the sense of $\cO_{K_2}$) for $1\leq j\leq e'$.

Since $F(x)$ is not Eisenstein over $K_2$, $c_{e'}$ must fail to be a uniformizing element. But $ev(c_{e'})=v(p)$ 
(in the sense of $\cO_{K_2}$), and $v(p)$ is the $e$th positive element of value group for $\cO_{K_2}$, 
so $c_{e'}$ does generate. So $K_1$ is relatively algebraically closed in $K_2$. This proves (1).

We now prove (2). The valuation ring of the induced valuation on $K_1$ is $K_1\cap \cO_{K_2}$, and its maximal 
ideal is $\cM_{K_2}\cap K_1$. By \cite[Theorem 4.1.3, pp.88]{EngPres}, Henselianity of a valued field 
is equivalent to the condition that 
any polynomial of the form
$$f:=X^n+X^{n-1}+a_{n-2}X^{n-2}+\dots+a_0$$
where all the coefficients $a_j$ are in the maximal 
ideal has a root in the field. So fix a polynomial $f$ as above with the condition that the coefficients $a_j$ are 
in the maximal ideal
$$\cM_{K_2}\cap K_1$$
of the induced valuation. 
Since all $a_j$ are in particular in $\cM_{K_2}$, by Henselianity of $K_2$ and 
\cite[Theorem 4.1.3, pp.88]{EngPres} we deduce that $f$ has a root $\alpha$ in $K_2$. 
Since by the first part, $K_1$ is relatively algebraically closed in $K_2$, this $\alpha$ must lie in 
$K_1$, and by another application of \cite[Theorem 4.1.3, pp.88]{EngPres} we deduce that $K_1$ is Henselian. The proof of the Lemma is complete. 
\end{proof}

We can now prove the following.

\begin{lem}\label{lem1} Let $K_1 \rightarrow K_2$ be an 
embedding of models of $Th(L)$. Then 
\begin{equation}\label{(*)}
\cO_{K_2}\cap K_1=\cO_{K_1}.
\end{equation}
\end{lem}
\begin{proof}
Consider the valuation ring in $K_1$ induced from 
$\cO_{K_2}$. By Lemma \ref{lem2}, it is Henselian. Since any two Henselian valuation rings in $K_1$ are comparable, 
and $K_1$ has rank one value group (since its value group is a $\Z$-group because it is elementarily equivalent to the 
value group of $L$), 
by \cite[Theorem 4.4.1]{EngPres} the induced valuation on $K_1$ must agree with 
that given by $\cO_{K_1}$ and \ref{(*)} follows.\end{proof}

It follows from Lemmas 1 and 2 that the valuation rings are $\forall_1$-definable uniformly for models of $Th(L)$.

\subsection{Model completeness for a finite extension of $\Q_p$}

In the case $K\equiv L$ and $[L:\Q_p] <\infty$, and in this case the multiplicative group of the 
residue field is isomorphic to the subgroup $\mu_{p^f-1}$ of $(p^f-1)$th roots of unity 
in $K^*$. If one has a cross-section 
$\Gamma \rightarrow K^*$, then $G_{K,0}$ is a subgroup of $K^*$, and in any case (with cross-section or not) 
it is elementarily equivalent to $\mu_{p^f-1} \times \Gamma$. Note that the $\mu_{p^f-1}$ factor is 
definable as the set of $(p^f-1)$-torsion elements. 

So fix such an $L$, with its attendant numbers $n,e,f$ with $n=ef$. For any field $L$ such that $K\equiv L$, 
the value group is a $\Z$-group, and $v(p)$ is the $e$th positive element of the value group. 

Now suppose $K_1\rightarrow K_2$ is an extension of models of $Th(L)$. Let $\gamma$ be a uniformizing 
parameter for $K_1$, i.e., $v(\gamma)$ is the least positive element if $v(K_1)$. By the preceding, 
$\gamma$ is also a uniformizing element for $v(K_2)$.

\begin{lem}\label{lem-ring} For any $k=mv(p)$, where $m\geq 0$, the embedding of local rings
$$\cO_{K_1,k} \rightarrow \cO_{K_2,k}$$
is elementary.\end{lem}
\begin{proof} For any $k=mv(p)$, where $m\geq 0$, the rings $\cO_{K_1,k}$ and $\cO_{K,k_2}$ have the same 
cardinality since $K_1$ and $K_2$ have the same finite residue field, so the inclusion 
$\cO_{K_1,k} \rightarrow \cO_{K_2,k}$ is an isomorphism, and hence is elementary.
\end{proof}

\begin{lem}\label{lem-group} For any $k=mv(p)$, where $m\geq 0$, the embedding of groups
$$G_{K_1,k} \rightarrow G_{K_2,k}$$ is elementary\end{lem}

\begin{remark} In general, the theory of the structure 
$\Z \times (torsion~subgroup)$ is not model-complete.\end{remark}

Now we give a new proof of model completeness for a finite extension of $\Q_p$. 
Let $L$ be a finite extension of $\Q_p$. Let $K_1 \rightarrow K_2$ be an embedding of models of 
$Th(L)$. We show that the embedding of $K_1$ in $K_2$ is elementary. Let $\varphi(\bar x)$ be an $\cL_{rings}$-formula 
and consider $\varphi(\bar a)$ where $\bar a$ is a tuple from $K_1$. By Theorem \ref{bk-thm}, there is a constant $N\geq 0$ 
and an $\cL_{BK}$-formula $\psi(\bar x)$ which is quantifier-free in the field sort 
such that
$$Th(L)\vdash \forall \bar x (\varphi(\bar x) \leftrightarrow \psi(\bar x)).$$
Since $K_1$ and $K_2$ are models of $Th(L)$, 
the formula $\forall \bar x (\varphi(\bar x) \leftrightarrow \psi(\bar x))$ holds in both $K_1$ and $K_2$. Hence 
$$K_i\models \varphi(\bar a) \leftrightarrow \psi(\bar a),$$
where $i=1,2$. 
The subformula of $\psi(\bar a)$ from the field sort is quantifier free and so will hold in $K_1$ if and only if it holds in $K_2$. 
Thus to prove that the inclusion of $K_1$ into $K_2$ is elementary, 
it suffices to consider the sub-formula of $\psi(\bar a)$ involving the sorts other than the field sort. 
In $K_i$ (for $i=1,2$), this formula is a Boolean combination 
of formulas of the sorts $\cO_{K_i,k}$, formulas of the sorts 
$G_{K_i,k}$, and formulas involving the relation $\Theta_k$ for finitely many values of $k$. We claim that each subformula of $\psi(\bar a)$ 
of each sort (including subformulas containing $\Theta_k$) holds in $K_1$ if and only if it holds in $K_2$. This would imply that 
$\psi(\bar a)$ holds in $K_1$ if and only if 
it holds in $K_2$, which implies that $\varphi(\bar a)$ holds in $K_1$ if and only if it holds in $K_2$. To prove the claim, 
by Lemmas \ref{lem-ring} and \ref{lem-group}, 
the embedding of rings $\cO_{K_1,k} \rightarrow \cO_{K_2,k}$ and the embedding of 
groups $G_{K_1,k}\rightarrow G_{K_2,k}$ are both 
elementary for $k=m.v(p)$ and any $m\geq 0$. Using the above interpretation of 
$(G_{K_i,k}\times \cO_{K_i,k}, \Theta_k)$ in $(G_{K_i,k}\times G_{K_i,k},\Theta^+_k)$ (for $i=1,2$), 
we deduce that the embedding 
$$(K_1,G_{K_1,k},\cO_{K_1,k},\Theta_k)\rightarrow (K_2,G_{K_2,k},\cO_{K_2,k},\Theta_k)$$
is elementary. This establishes the claim, and completes the proof.




\bibliographystyle{amsplain}
\bibliography{anbib}
\end{document}